%% file: main.tex
\documentclass[conference]{IEEEtran}
\IEEEoverridecommandlockouts
\usepackage[numbers,sort&compress]{natbib}
\usepackage{amsmath,amssymb,amsfonts}
\usepackage{graphicx}
\usepackage{textcomp}
\usepackage{xcolor}
\usepackage{algorithm}
\usepackage{algpseudocode}
\usepackage{booktabs}
\usepackage{amsmath,amssymb,amsfonts}
\usepackage{amsthm} 
\newtheorem{theorem}{Theorem}[section]

\newtheorem{lemma}[theorem]{Lemma}
\usepackage{enumitem} 
\usepackage{mathtools}
\usepackage{bm}
\usepackage{url}
\newtheorem{assumption}{Assumption}

\begin{document}

\title{Scalable Mixed-Integer Optimization with Neural Constraints via Dual Decomposition}

\author{\IEEEauthorblockN{
Shuli Zeng\IEEEauthorrefmark{1},
Sijia Zhang\IEEEauthorrefmark{1},
Feng Wu\IEEEauthorrefmark{1},
Shaojie Tang\IEEEauthorrefmark{2},
Xiangyang Li\IEEEauthorrefmark{1}}
\IEEEauthorblockA{\IEEEauthorrefmark{1}\textit{University of Science and Technology of China}, Hefei, China\\
\texttt{zengshuli0130@mail.ustc.edu.cn}, \texttt{sxzsj@mail.ustc.edu.cn},\\
\texttt{wufeng02@ustc.edu.cn}, \texttt{xiangyangli@ustc.edu.cn}}
\IEEEauthorblockA{\IEEEauthorrefmark{2}\textit{State University of New York at Buffalo}, Buffalo, NY, USA\\
\texttt{shaojiet@buffalo.edu}}
\thanks{Corresponding authors: Sijia Zhang, Feng Wu, and Xiangyang Li.}
}

\maketitle

\begingroup
\renewcommand\thefootnote{}\footnotetext{%
\textbf{Authors’ extended version.} Accepted to AAAI-26 (Main Technical Track, Oral).
Please cite the conference version. Full appendices and additional experiments are included here.}
\addtocounter{footnote}{-1}
\endgroup

\begin{abstract}
Embedding deep neural networks (NNs) into mixed-integer programs (MIPs) is attractive for decision making with learned constraints, yet state-of-the-art “monolithic” linearisations blow up in size and quickly become intractable. In this paper, we introduce a novel dual-decomposition framework that relaxes the single coupling equality $u=x$ with an augmented Lagrange multiplier and splits the problem into a vanilla MIP and a constrained NN block. Each part is tackled by the solver that suits it best—branch \& cut for the MIP subproblem, first-order optimisation for the NN subproblem—so the model remains modular, the number of integer variables never grows with network depth, and the per-iteration cost scales only linearly with the NN size. On the public \textsc{SurrogateLIB} benchmark, our method proves \textbf{scalable}, \textbf{modular}, and \textbf{adaptable}: it runs \(120\times\) faster than an exact Big–\(M\) formulation on the largest test case; the NN sub-solver can be swapped from a log-barrier interior step to a projected-gradient routine with no code changes and identical objective value; and swapping the MLP for an LSTM backbone still completes the full optimisation in \(47\)s without any bespoke adaptation.
\end{abstract}


\input{sections/1introduction}
\input{sections/2background}
\input{sections/3motivation}
\input{sections/4method}
\input{sections/5experiments}

\input{sections/6conclusion}

\section*{Acknowledgments}
The research is partially supported by  Quantum Science and Technology-National Science and Technology Major Project (QNMP) 2021ZD0302900 and China National Natural Science Foundation with No. 62132018, 62231015, "Pioneer" and "Leading Goose" R\&D Program of Zhejiang", 2023C01029, and 2023C01143.
\bibliographystyle{IEEEtranN}
\bibliography{aaai}
\clearpage
\appendix
\input{sections/8appendix}
\end{document}

%% file: sections/1introduction.tex
\section{Introduction}
\label{sec:introduction}

Intelligent decision systems increasingly integrate neural networks into decision-making and optimization pipelines~\cite{bengio2021machine,cappart2021combining,joshi2020machine}. Neural networks provide data–driven surrogates that capture complex, nonlinear dependencies, and their outputs can be enforced as constraints inside mixed-integer programming (MIP) models while preserving the latter’s combinatorial guarantees (Fig.~\ref{fig:nn-embedded-mip}). This fusion marries the predictive accuracy of data-driven models with the exactness of discrete optimisation: the NN supplies rich, sample-efficient knowledge of latent physics or economics~\cite{bertsimas2016analytics,misaghian2022assessment,jalving2023beyond}, while the MIP layer delivers globally optimal and certifiably feasible decisions.  As a result, practitioners can exploit learned structure without sacrificing safety guarantees or the ability to produce worst-case certificates—an advantage unattainable by black-box heuristics~\cite{droste2006upper} alone.

Prior research~\cite{ceccon2022omlt,bergman2022janos,lombardi2017empirical,maragno2025mixed} has taken important steps toward incorporating neural networks as mathematical constraints. One line of work focuses on \textbf{exact, monolithic integration}: encoding the network’s operations (e.g., ReLU activations) directly into a MIP or satisfiability formulation.  \citet{tjeng2017evaluating} formulate the verification of piecewise-linear neural networks as a mixed-integer linear program, enabling exact adversarial robustness checks with orders-of-magnitude speedups over earlier methods. Building on such encodings, \citet{anderson2020strong} developed stronger MILP formulations for ReLU-based networks, which further tighten the linear relaxations and improve solving efficiency. In parallel, practical frameworks have emerged to automate these integrations: for example, PySCIPOpt-ML \cite{turner2023pyscipopt} provides an interface to embed trained machine learning models (including neural nets) as constraints within the open-source solver SCIP~\cite{bestuzheva2023enabling}. These advances demonstrate the viability of treating a neural network as a combinatorial constraint, allowing exact optimization or formal verification involving network outputs.

\begin{figure}[t]
  \centering
  \includegraphics[width=0.47\textwidth]{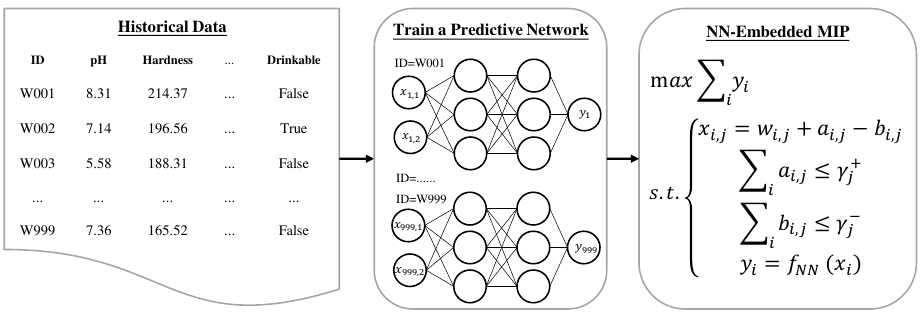}
  \caption{Workflow of the \textbf{NN‑Embedded MIP}: historical samples are adjusted under budget constraints, evaluated by a neural classifier, and optimised via a mixed‑integer solver.}
  \label{fig:nn-embedded-mip}
\end{figure}

However, existing approaches to neural network integration face fundamental limitations. A prevailing tactic is to \emph{fully linearise} every activation, yielding a single, rigid MILP. Any change to the network then forces a complete reformulation; the encoding relies on large Big–$M$ bounds~\cite{grimstad2019relu,huchette2023deep,badilla2023computational} that weaken relaxations~\cite{griva2008linear,camm1990cutting}; and runtime soars with depth and width. In our largest test, translating a $64{\times}512$ ReLU model for 500 samples with \textsc{PySCIPOpt-ML} inflates the problem from 500 variables to 5,900,000 variables and 7,800,000  constraints, yet SCIP still times out after two hours. Even the tighter formulations of \citet{anderson2020strong} remain highly size-sensitive. These limitations motivate a framework that \textbf{decouples} the neural and combinatorial parts, achieving \textbf{scalable}, \textbf{modular}, and \textbf{adaptable} integration without loss of optimality.

We propose a \textbf{dual-decomposition} scheme that duplicates the integer vector \(x\) with a continuous copy \(u\), relaxes the coupling equality \(u=x\) via an augmented-Lagrange term, and then alternates between two native solves:  
(i) a branch-and-cut MIP on \(x\) and (ii) a constrained first-order update on \(u\).  
After each pair of solves the dual multipliers are updated, shrinking the mismatch until the two blocks agree; because neither block is ever reformulated into the language of the other, the method keeps the full power of specialised MILP engines and modern NN optimisers while sidestepping the Big-\(M\) explosion. 
This method yields three validated advantages.  
\textbf{Scalability}: On the X-Large case (6×128 ReLU, 500 samples) our dual-decomposition solver finishes in 8.7s, whereas \textsc{PySCIPOpt-ML} needs 900s or more—a \(100\times\) speed-up obtained.  
\textbf{Modularity}: replacing the NN sub-solver (projected-gradient $\leftrightarrow$ log-barrier) leaves the MIP untouched and changes total time by less than a factor of two.  
\textbf{Adaptability}: swapping the backbone MLP for a CNN or LSTM requires no reformulation and still converges in under 47s with the same objective value.  
Section~\ref{sec:experiments} details these results, confirming that dual decomposition can combine learning surrogates with discrete optimisation at scale.

Our main contributions are as follows:
\begin{itemize}
\item \textbf{Algorithm.} We introduce a dual-decomposition framework that couples a mixed-integer subproblem with a neural subproblem through an augmented-Lagrange split, avoiding any Big–\(M\) reformulation.
\item \textbf{Theory.} We proof global convergence of the alternating scheme and prove that the per-iteration cost grows only linearly with network size, guaranteeing scalability by design.
\item \textbf{Experiments.} Extensive experiments confirm that the solver is \emph{scalable} (up to \(100\times\) faster than monolithic baselines), \emph{modular} (sub-solvers can be swapped without code changes), and \emph{adaptable} (CNN and LSTM backbones solve in \(<47\) s with identical optima).
\end{itemize}

%% file: sections/2background.tex
\section{Preliminaries}

\paragraph{Mixed Integer Linear Programming (MILP).}
Given a set of decision variables $\bm{x}\in \mathbb{R}^n$, the MILP problem is formulated as follow:
\begin{equation}
\label{equ:milp}
\begin{aligned}
    \min \quad &\bm{c}^{\top} \bm{x}, \\
    \mathrm{s.t.} \quad &\bm{A} \bm{x} \geq \bm{b}, \quad \bm{l} \leq \bm{x} \leq \bm{u}, \\
    &\bm{x} \in \{0,1\}^p \times \mathbb{Z}^q \times \mathbb{R}^{n-p-q},
\end{aligned}
\end{equation} 
where $\bm{c} \in \mathbb{R}^n$ is the objective coefficients, $\bm{A} \in \mathbb{R}^{m \times n}$ is the constraint coefficient matrix, $\bm{b} \in \mathbb{R}^m$ is the right-hand side vector, $\bm{l},\bm{u} \in \mathbb{R}^n$ are the variable bounds. 

\paragraph{Augmented Lagrangian Method (ALM).} 
\label{subsec:augLag}
We consider the equality‑constrained problem as below:
\begin{align}
  \min_{x\in\mathbb{R}^{n}}\; f(x)  ~~~\text{s.t.}~~~ c_i(x)=0,\quad i\in\mathcal{E},
  \label{eq:al_prob}
\end{align}
where $\mathcal{E}$ indexes the set of equality constraints.  
A straightforward quadratic‑penalty approach enforces the constraints by minimising the following equation: 
\begin{equation}
  \Phi_k^{\text{pen}}(x)
  \;=\;
  f(x)
  + \mu_k \sum_{i\in\mathcal{E}} c_i(x)^2,
  \label{eq:penalty}
\end{equation}
and gradually driving $\mu_k\!\to\!\infty$; however, large penalties typically induce severe ill‑conditioning~\cite{bertsekas2014constrained}.

To mitigate this, the augmented Lagrangian augments the objective with a linear multiplier term as:
\begin{equation}
  \Phi_k(x)
  \;=\;
  f(x)
  + \tfrac{\mu_k}{2}\sum_{i\in\mathcal{E}} c_i(x)^2
  + \sum_{i\in\mathcal{E}} \lambda_i\,c_i(x),
  \label{eq:augLag}
\end{equation}
where $\lambda\in\mathbb{R}^{|\mathcal{E}|}$ estimates the true Lagrange multipliers.  
At outer iteration $k$, one solves:
\[
  x^{(k)} \;=\; \arg\min_{x}\Phi_k(x),
\]
then updates the multipliers by:
\begin{equation}
  \lambda_i^{(k+1)}
  \;=\;
  \lambda_i^{(k)} + \mu_k\,c_i\!\bigl(x^{(k)}\bigr),
  \qquad i\in\mathcal{E},
  \label{eq:lambdaUpdate}
\end{equation}
and increases $\mu_k$ (e.g.\ $\mu_{k+1}=\beta\mu_k$ with $\beta>1$) only when constraint violations stagnate.  
Because the linear term offsets part of the quadratic penalty, $\mu_k$ can stay moderate, leading to better numerical conditioning and stronger convergence guarantees. In practice, safeguards that bound $\lambda$ are often used to prevent numerical overflow~\cite{hestenes1969multiplier}.

%% file: sections/3motivation.tex
\section{Problem Formulation and Motivation}
\label{sec:motivation}

We study mixed-integer programmes (MIPs) that \emph{embed} a pretrained neural network (NN) directly into their objective and constraints. Let \(x\!\in\!\mathbb{Z}^{p}\) denote the discrete decisions and \(y:=f_{\theta}(x)\!\in\!\mathbb{R}^{q}\) the NN output;  
\(f_{\theta}\) is differentiable with fixed parameters \(\theta\).  
The general problem is
\begin{equation}
\label{eq:general-nn-mip}
\begin{aligned}
\min_{x\in\mathbb{Z}^{p}}\quad 
    & c^{\top}x + d^{\top}y \\[2pt]
\text{s.t.}\quad 
    & A\!\begin{bmatrix}x\\ y\end{bmatrix} \le b, \\[4pt]
    & y = f_{\theta}(x).
\end{aligned}
\end{equation}
with  
\(c\!\in\!\mathbb{R}^{p}\) and \(d\!\in\!\mathbb{R}^{q}\) weighting the discrete cost and the NN-dependent cost,  
\(A\!\in\!\mathbb{R}^{m\times(p+q)}\) and \(b\!\in\!\mathbb{R}^{m}\) capturing all affine limits that jointly constrain the decisions and the NN prediction.  
Because the mapping \(x \mapsto y=f_{\theta}(x)\) is nonlinear, problem~\eqref{eq:general-nn-mip} inherits both the combinatorial hardness of MIPs and the nonlinearity of deep networks, making it particularly challenging to solve.

\paragraph{Limitations of Network-to-MIP Embedding.}
A common strategy for incorporating a pretrained NN into a MIP is to replace each nonlinear operation by piecewise-linear surrogates—e.g.\ Big-M formulations~\cite{cococcioni2021big} or SOS1 constraints~\cite{fischer2018branch}.  Although exact, these encodings introduce \(\mathcal{O}(Lh)\) additional binaries and constraints for an \(L\)-layer, \(h\)-unit MLP, which can blow up the solver’s search tree even for moderate network sizes.  Moreover, available linearization routines support only basic primitives (dense ReLUs, simple max-pooling), whereas modern architectures—convolutional filters, recurrent gates, attention blocks and layer-normalisation—have no off-the-shelf encodings.  Adapting each novel layer requires bespoke linearization, increasing modelling complexity and hindering extensibility. 


\paragraph{Rationale for Decomposition.}
By contrast, treating the NN as a black-box oracle avoids this overhead but sacrifices explicit feasibility guarantees, risking violations of critical constraints in safety-sensitive applications. To overcome the computational challenges associated with monolithic linearization approaches, we propose employing \emph{augmented Lagrangian decomposition}. This technique partitions the original problem~\eqref{eq:general-nn-mip} into simpler subproblems, each handled by specialized solvers, while iteratively maintaining solution consistency through dual updates. This decomposition approach provides significant computational advantages, improved scalability, and robust theoretical convergence properties.

%% file: sections/4method.tex
\section{Our Decomposition-Based Method}
We introduce an auxiliary continuous variable $u \in \mathbb{R}^p$ to duplicate the integer variable $x$, resulting in an equivalent but more tractable formulation featuring explicit coupling constraints:
\begin{equation}
\label{eq:decomposed-formulation}
\begin{aligned}
  \min_{x,u}\quad
      & c^{\top}x + d^{\top}f_{\theta}(u) \\[4pt]
  \text{s.t.}\quad
      & A_{\text{MIP}}\,x \;\le\; b_{\text{MIP}}, \\[2pt]
      & A_{\text{NN}}
        \begin{bmatrix}
          u \\[2pt] f_{\theta}(u)
        \end{bmatrix}
        \;\le\; b_{\text{NN}}, \\[2pt]
      & u \;=\; x, 
        \qquad x \in \mathbb{Z}^{p}.
\end{aligned}
\end{equation}

Here $A_{\text{MIP}}\!\in\!\mathbb{R}^{m_1\times p},\; b_{\text{MIP}}\!\in\!\mathbb{R}^{m_1}$ store the “legacy’’ linear limits that involve only the integer vector $x$.  
All remaining affine constraints that depend on the NN output are written as  
$A_{\text{NN}}\!\in\!\mathbb{R}^{m_2\times(p+q)}$ and $b_{\text{NN}}\!\in\!\mathbb{R}^{m_2}$ acting on the stacked vector $\bigl[u;\,f_{\theta}(u)\bigr]$.  
Across iterations $u$ stays continuous—enabling inexpensive first-order updates of the NN block—while the coupling equation $u=x$ is enforced at convergence to recover a valid integer solution.

\subsection{High-Level Decomposition Approach}
The key theoretical challenge we address is the nonconvex equality constraint $u=x$. We handle this complexity by introducing dual multipliers (Lagrange multipliers) $\lambda\in\mathbb{R}^p$ and a quadratic penalty term controlled by parameter$\rho>0$. The resulting augmented Lagrangian function is given by:
\begin{equation}
\label{eq:aug-lag}
\begin{aligned}
  \mathcal{L}_{\rho}(x,u,\lambda)
  \;=\;
      &\, c^{\top}x
      + d^{\top}f_{\theta}(u) \\[2pt]
      &+\, \lambda^{\top}(u - x)
      + \frac{\rho}{2}\,\|u - x\|_{2}^{2} .
\end{aligned}
\end{equation}
The term $\lambda^{\top}(u-x)$ is the \emph{classical} Lagrange coupling: at a Karush-Kuhn-Tucker(KKT) point~\cite{gordon2012karush} it drives the equality $u=x$ by adjusting the dual multipliers, while supplying first-order information to each subproblem.  
The quadratic add-on $\tfrac{\rho}{2}\|u-x\|_{2}^{2}$ augments the Lagrangian, injecting strong convexity that damps oscillations and markedly accelerates the alternating updates observed in practice.  
The penalty weight $\rho>0$ balances two competing objectives: (i)~allowing the MIP and NN blocks enough flexibility to explore their individual feasible regions (small~$\rho$), and (ii)~tightening the coupling to secure rapid primal feasibility (large~$\rho$).  
Following standard augmented-Lagrangian heuristics~\cite{larsson2004augmented}, we initialize $\rho$ conservatively and increase it adaptively whenever the primal residual $\|u-x\|_2$ ceases to decrease, thereby achieving both robustness and efficiency.
Figure~\ref{fig:dual-decomposition} illustrates our proposed dual decomposition workflow, highlighting the iterative, coordinated solving of the decomposed subproblems. Algorithm \ref{alg:dual_decomp} provides the pseudocode.

\begin{figure}[t]
    \centering
    \includegraphics[width=0.9\linewidth]{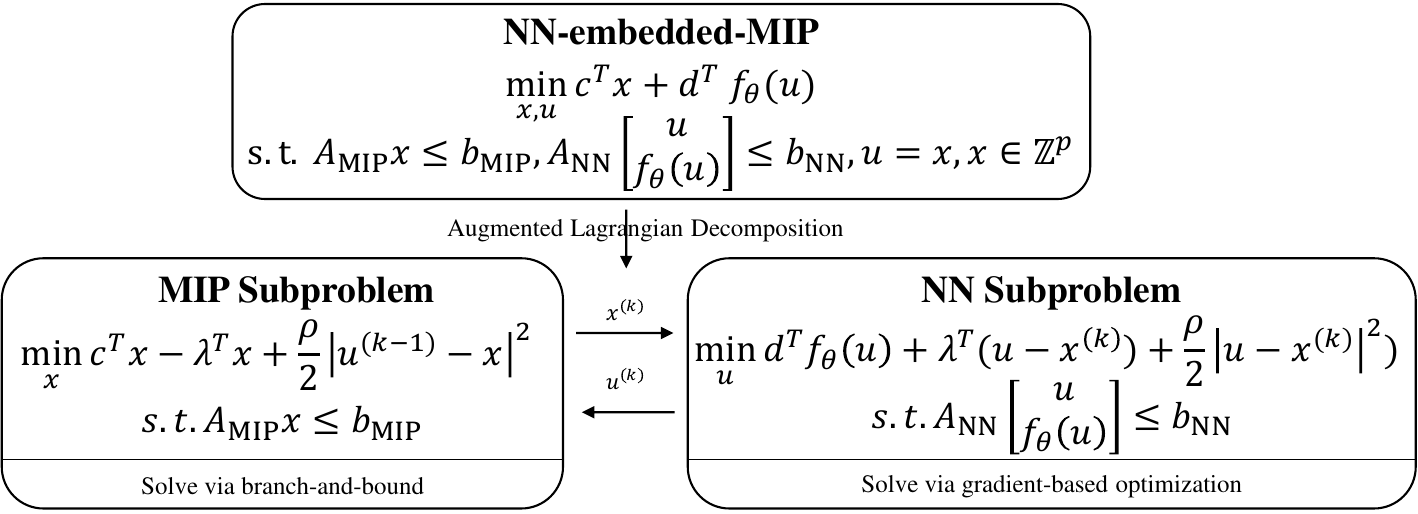}
    \caption{Illustration of the proposed NN-embedded MIP solution framework using augmented Lagrangian decomposition, emphasizing iterative primal-dual coordination.}
    \label{fig:dual-decomposition}
\end{figure}

\begin{algorithm}[htbp]
\caption{Dual Decomposition for NN-Embedded MIP}
\label{alg:dual_decomp}
\begin{algorithmic}[1]
\Require Initialization $u^{(0)}, \lambda^{(0)}, \rho>0$, tolerance $\varepsilon$, maximum iterations $K$
\For{$k=1,\dots,K$}
    \State Solve MIP subproblem for $x^{(k)}$ to update discrete decisions.
    \State Solve NN subproblem (PGD or Log-Barrier) for $u^{(k)}$ to refine continuous inputs.
    \State Update dual multiplier: $\lambda^{(k)} \leftarrow \lambda^{(k-1)} + \rho(u^{(k)}-x^{(k)})$.
    \State Check convergence: if $\|u^{(k)} - x^{(k)}\|_\infty < \varepsilon$, terminate.
    \State Adjust penalty parameter $\rho$ adaptively based on residuals.
\EndFor
\State \Return optimal solution pair $(x^{\star},u^{\star})$.
\end{algorithmic}
\end{algorithm}

\subsection{Detailed Modular Subproblem Decomposition}
We then decompose the augmented Lagrangian into two tractable subproblems solved alternately and coordinated by iterative updates of the dual multipliers:

\paragraph{MIP Subproblem.}
Given $(u^{(k-1)},\lambda^{(k-1)})$, the MIP subproblem is defined as:
\begin{equation}
\label{eq:mip-update}
\begin{split}
  x^{(k)}
  \;=\; \arg\min_{x \in \mathbb{Z}^{p}} 
  & \;\Bigl(
       c^{\top}x
       - \lambda^{(k-1)\!\top} x
       + \frac{\rho}{2}\,\|u^{(k-1)} - x\|_{2}^{2}
     \Bigr) \\
  & \text{s.t.}\quad
  A_{\text{MIP}}\,x \;\le\; b_{\text{MIP}} .
\end{split}
\end{equation}
This formulation retains the linear constraints $A_{\text{MIP}}\,x \;\le\; b_{\text{MIP}}$ but adjusts the objective to include the penalty term $\|u^{(k-1)} - x\|_2^2$, thereby encouraging proximity between $x$ and the current NN solution $u^{(k-1)}$. It includes the linear term $-\lambda^{(k-1)\!\top}x$, the Lagrange multiplier contribution originating from the coupling constraint $u=x$ in the augmented Lagrangian. The companion constant $\lambda^{(k-1)\!\top}u^{(k-1)}$ is omitted because it is independent of $x$ and thus irrelevant for the minimisation. 
Modern branch-and-bound or cutting-plane solvers~\cite{bestuzheva2021scip,gurobi} can efficiently handle this subproblem due to its convexified structure~\cite{mistry2021mixed}.

\paragraph{NN Subproblem.}
Given the discrete solution $x^{(k)}$ and multiplier $\lambda^{(k-1)}$, the continuous variables are optimized through solving:
\begin{equation}
\label{eq:nn-update}
\begin{split}
  u^{(k)}
  \;=\; \arg\min_{u \in \mathbb{R}^{p}}
     \Bigl\{ \;
       & d^{\top}f_{\theta}(u) 
       +\,\lambda^{(k-1)\!\top}\bigl(u - x^{(k)}\bigr) \\[2pt]
       & +\,\dfrac{\rho}{2}\,\|u - x^{(k)}\|_{2}^{2}
     \;\Bigr\} \\[6pt]
  \text{s.t.}\quad
    & A_{\text{NN}}
      \begin{bmatrix}
        u \\[2pt] f_{\theta}(u)
      \end{bmatrix}
      \;\le\; b_{\text{NN}} .
\end{split}
\end{equation}
Fixing the discrete iterate $x^{(k)}$ isolates the continuous decision~$u$, letting us recalibrate the neural input without perturbing the combinatorial block.  
The first term $d^{\top}f_{\theta}(u)$ transmits the original cost coefficients to the NN, while the dual inner product $\lambda^{(k-1)\!\top}(u - x^{(k)})$ conveys the mismatch detected in the preceding coordination step.  
The proximal penalty $\tfrac{\rho}{2}\|u - x^{(k)}\|_2^{2}$ tempers aggressive updates, yielding a well-conditioned landscape.  
Crucially, the stacked linear constraint
$A_{\text{NN}} \left[u \; ; \; f_{\theta}(u)\right] \le b_{\text{NN}}$
preserves every safety or performance limit originally imposed on the network output, ensuring that each neural adjustment remains admissible before the next dual synchronization.

\subsection{Solving the NN Subproblem with Constraints}
The NN subproblem is non-trivial due to the nonlinear nature of
$f_{\theta}(u)$ and the constraint
$A_{\text{NN}}\!\left[\,u \; ; \; f_{\theta}(u)\right] \le b_{\text{NN}}$,
which induces the feasible set
\[
  \mathcal{C}
  \;:=\;
  \Bigl\{
    u\in\mathbb{R}^{p}
    \;\Bigm|\;
    A_{\text{NN}}
    \begin{bmatrix}
      u \\[2pt] f_{\theta}(u)
    \end{bmatrix}
    \le b_{\text{NN}}
  \Bigr\}.
\]

\paragraph{Projected Gradient Descent (PGD).}
When projection onto $\mathcal{C}$ is affordable—e.g., if $A_{\text{NN}}$ encodes simple box or monotone limits—we perform
\begin{enumerate}
\item a gradient step on the smooth part of the objective, then
\item an orthogonal projection back onto $\mathcal{C}$.
\end{enumerate}
Let $J_{f_{\theta}}(u)$ denote the Jacobian of $f_{\theta}$ at~$u$.  
The gradient of the augmented Lagrangian with fixed $(x^{(k)},\lambda^{(k-1)})$ reads
\[
\nabla_{u}\mathcal{L}_{\rho}(u)
  = J_{f_{\theta}}(u)^{\!\top}d
    + \lambda^{(k-1)}
    + \rho\bigl(u - x^{(k)}\bigr).
\]
With stepsize $\eta>0$ we update
\[
u \leftarrow \operatorname{Proj}_{\mathcal{C}}
       \!\bigl(u - \eta\nabla_{u}\mathcal{L}_{\rho}(u)\bigr),
\]
where $\operatorname{Proj}_{\mathcal{C}}$ is the Euclidean projection~\cite{liu2009efficient} onto~$\mathcal{C}$. 
The proximal term $\rho\|u-x^{(k)}\|_2^{2}$ ensures the landscape is well-conditioned, and PGD typically converges in a handful of iterations~\cite{haji2021comparison}.

\paragraph{Log-Barrier Interior Method.}
If projection is itself expensive (e.g.\ $\mathcal{C}$ is defined by many affine cuts on $f_{\theta}$), we instead absorb the constraints via a logarithmic barrier:
\[
\begin{aligned}
\min_{u \in \mathbb{R}^{p}}\;\Bigl\{\,%
    & d^{\top} f_{\theta}(u) +\, \lambda^{(k-1)\!\top}\!\bigl(u - x^{(k)}\bigr) +\, \dfrac{\rho}{2}\,\|u - x^{(k)}\|_{2}^{2} \\[4pt]
    &-\; \mu
       \sum_{j=1}^{m}
       \log\!\Bigl(
         b_{\text{NN},j}
         - a_{j}^{\top}
           \begin{bmatrix}
             u \\[2pt] f_{\theta}(u)
           \end{bmatrix}
       \Bigr)
\,\Bigr\}
\end{aligned}
\]
where $a_{j}^{\top}$ is the $j$-th row of $A_{\text{NN}}$ and $\mu>0$ is the barrier parameter.  
Starting from a strictly feasible point, we solve the above with a second-order method and gradually decrease $\mu$ ($\mu\!\leftarrow\!0.1\mu$) until the dual feasibility tolerance is reached.  
This interior-point strategy dispenses with explicit projections yet retains strict feasibility along the entire Newton path \cite{toussaint2017tutorial}; it is especially beneficial when $\mathcal{C}$ is described by many coupled linear cuts but occupies a relatively low-volume region, where projection costs dominate.

\subsection{Dual Updates and Penalty Parameter Adjustments}
The dual multipliers are iteratively updated to maintain primal-dual feasibility:
\begin{equation}
\lambda^{(k)} \leftarrow \lambda^{(k-1)} + \rho(u^{(k)} - x^{(k)}).
\end{equation}
The penalty parameter $\rho$ is dynamically adjusted to effectively balance primal feasibility and dual optimization stability:
\begin{itemize}
\item \textbf{Increasing} $\rho$ enhances primal feasibility enforcement when constraint violations persist.
\item \textbf{Decreasing} $\rho$ accelerates convergence by mitigating slow dual updates or numerical instability.
\end{itemize}

\section{Theoretical Analysis}
In this section, we provide two theoretical perspectives on our proposed decomposition framework: (i) a convergence analysis showing that our primal-dual scheme yields globally convergent iterates under standard assumptions; and (ii) a scalability analysis demonstrating that the computational complexity of our method grows only linearly with the neural network size, ensuring robust performance as problem instances scale.

\subsection{Convergence Analysis}
We prove that Algorithm~\ref{alg:dual_decomp} converges to a KKT point of the augmented Lagrangian.
\begin{assumption}
\label{ass:regularity}
\begin{enumerate}
\item $\mathcal{X}:=\{x\in\mathbb{Z}^{p}\mid A_{\text{MIP}}x\le b_{\text{MIP}}\}$ is non-empty and bounded;  
\item The map $g(u):=d^{\top}f_{\theta}(u)$ is convex and has $L$-Lipschitz gradient ($L$–smooth);
\item $\mathcal{C}:=\{u\mid A_{\text{NN}}[u;f_{\theta}(u)]\le b_{\text{NN}}\}$ is non-empty and satisfies Slater’s condition;  
\item A penalty $\rho>L$ is chosen once (or increased finitely many times and then fixed).
\end{enumerate}
\end{assumption}

\begin{theorem}[Convergence to a KKT point]
\label{thm:stationary}
Under Assumption~\ref{ass:regularity}, the sequence $\{(x^{(k)},u^{(k)},\lambda^{(k)})\}$ generated by Algorithm~\ref{alg:dual_decomp} is bounded. Every accumulation point $(x^{\star},u^{\star},\lambda^{\star})$ satisfies:
\begin{equation}
\label{eq:kkt-conditions}
\begin{aligned}
  & u^{\star} \;=\; x^{\star}, \\[4pt]
  & \nabla_{u}\mathcal{L}_{\rho}\!\bigl(x^{\star},u^{\star},\lambda^{\star}\bigr)
    \;=\; 0, \\[4pt]
  & x^{\star} \in
    \arg\min_{x \in \mathcal{X}}
      \mathcal{L}_{\rho}\!\bigl(x,u^{\star},\lambda^{\star}\bigr).
\end{aligned}
\end{equation}
Moreover,
\begin{equation}
\label{eq:lagrangian-rate}
\begin{aligned}
  \min_{t \le k}\;\|u^{(t)} - x^{(t)}\|_{2}^{2}
    &= \mathcal{O}\!\bigl(\tfrac{1}{k}\bigr), \\[4pt]
  \mathcal{L}_{\rho}\!\bigl(x^{(k)},u^{(k)},\lambda^{(k)}\bigr)
    - \mathcal{L}_{\rho}\!\bigl(x^{\star},u^{\star},\lambda^{\star}\bigr)
    &= \mathcal{O}\!\bigl(\tfrac{1}{k}\bigr).
\end{aligned}
\end{equation}
\end{theorem}

\begin{proof}[Proof]
The argument proceeds in three steps. (i) Descent in the $u$–block. Because $g(u)=d^{\top}f_{\theta}(u)$ is $L$–smooth and $\rho>L$, the 
function $g(u)+\tfrac{\rho}{2}\|u-x\|_{2}^{2}$ is 
$(\rho-L)$-strongly convex in $u$.  Solving the NN subproblem exactly therefore yields a strict decrease in the augmented merit and a bound of order $\|u^{(k)}-x^{(k-1)}\|_{2}^{2}$.  (ii) Descent in the $x$–block. The discrete set $\mathcal{X}$ is finite and each MIP subproblem is solved to global optimality, so minimisation over $x$ never increases the augmented Lagrangian. (iii) Dual update. Updating $\lambda^{(k+1)}=\lambda^{(k)}+\rho\bigl(u^{(k)}-x^{(k)}\bigr)$ produces the telescoping identity 
\(
\mathcal{M}_{\rho}^{(k+1)}=
\mathcal{M}_{\rho}^{(k)}-\tfrac{\rho-L}{2}\|u^{(k)}-x^{(k)}\|_{2}^{2},
\)
so $\sum_{k}\|u^{(k)}-x^{(k)}\|_{2}^{2}<\infty$ and $\min_{t\le k}\|u^{(t)}-x^{(t)}\|_{2}^{2}=\mathcal{O}(1/k)$.  The merit is bounded below, hence the iterate sequence is bounded; by Bolzano–Weierstrass it has accumulation points, and the first-order optimality of both exact subsolvers implies every such point satisfies the KKT system. Full details are deferred to Appendix. \qedhere
\end{proof}

\subsection{Scalability with Neural Network Size}
\label{subsec:scale}
Next, we analyze how the complexity of our method scales with the neural network architecture.

\begin{theorem}[Linear Scalability in Network Size]
\label{thm:scalability}
Let $P$ denote the number of parameters in the neural network $f_{\theta}$. Then each iteration of Algorithm~\ref{alg:dual_decomp} has overall complexity
\begin{equation}
\mathcal{O}\bigl(T_{\text{MIP}} + P\bigr),
\end{equation}
where $T_{\text{MIP}}$ is the complexity of solving the MIP subproblem, independent of $P$. Consequently, for fixed MIP size, the runtime per iteration grows \emph{linearly} with the network size.
\end{theorem}

\begin{proof}[Proof]
(i) The MIP subproblem~\eqref{eq:mip-update} depends only on $x$, $\lambda$, and $u^{(k-1)}$; its size does not increase with $P$. Thus, $T_{\text{MIP}}$ is independent of the neural architecture.  
(ii) The NN subproblem~\eqref{eq:nn-update} requires evaluating $f_{\theta}(u)$ and its Jacobian $J_{f_\theta}(u)$. For standard feedforward networks, forward and backward passes both require $\mathcal{O}(P)$ time~\cite{goodfellow2016deep}.  
(iii) The projection or barrier operations involve only affine constraints $A_{\text{NN}}$, whose size is fixed by the problem specification and independent of $P$. Therefore, the per-iteration complexity is $\mathcal{O}(T_{\text{MIP}}+P)$, implying linear scaling with the network size. 
\end{proof}

In a Big-$M$ linearisation, every hidden unit contributes a constant number of \emph{additional variables and constraints}, so the model dimension itself grows roughly linearly with network depth and width. However, the resulting MILP remains \textsc{NP}-hard; empirical and worst-case studies show that solver time rises super-linearly, often exponentially, with that dimension increase~\cite{bixby1992implementing}.  By leaving the network outside the MILP, our decomposition sidesteps this blow-up and retains per-iteration linear scaling.


%% file: sections/5experiments.tex
\section{Experimental Evaluations}
\label{sec:experiments}

\noindent\textbf{Overview of Experiments.} We perform five complementary studies to thoroughly evaluate the proposed dual-decomposition framework: \textbf{(E1)} a head-to-head comparison with Big-M linearisation; \textbf{(E2)} a stress test that scales network depth and width; \textbf{(E3)} an ablation that removes the dual-update loop; \textbf{(E4)} a modularity study that swaps heterogeneous network architectures without changing solver logic; and \textbf{(E5)} a subsolver comparison between projected-gradient descent and a log-barrier interior method. Together, these experiments probe solution quality, runtime efficiency, robustness, architecture agnosticism, and the impact of continuous-block solver choice.\par

\subsection{Experimental Setup}
\subsubsection{Implementation Details.}
All experiments are implemented in \texttt{Python}: neural classifiers are trained with \texttt{PyTorch}~\cite{paszke2019pytorch}, while MIP subproblems are solved via \texttt{PySCIPOpt 8.0.4}~\cite{MaherMiltenbergerPedrosoRehfeldtSchwarzSerrano2016}. For the \textbf{linearisation baseline}, we use \textsc{PySCIPOpt-ML} to linearise the trained classifier via its Big-M \texttt{add\_predictor\_constr} interface. For \textbf{dual decomposition~(DD)}, we adopt an ADMM-style scheme with $\rho=10.0$, step size $\eta=0.01$, and convergence tolerance $\varepsilon=10^{-4}$. Each augmented-Lagrange iteration first solves the MIP subproblem, then refines the neural block—performing 25 Adam~\cite{kingma2014adam} steps in the PGD variant or an L-BFGS-B~\cite{zhu1997algorithm} update in the log-barrier variant. Experiments~E1--E4 solve the neural subproblem with the PGD variant. We allow at most 50 iterations and enforce a per-instance time budget of $300\,\text{s}$. Every configuration is executed under five random seeds $\{42,123,456,789,1024\}$ and we report averages. All experiments are conducted on a single machine equipped with eight NVIDIA GeForce RTX 4090 GPUs and two AMD EPYC 7763 CPUs.

\subsubsection{Dataset.}
\textbf{(1) Water Potability.}  Used in Experiments~E1--E4. The dataset offers ne water-quality features with binary drinkability labels.From the undrinkable subset we draw instances of size $n\in\{25,50,100,200,400\}$. Each optimisation task selects feature adjustments---subject to per-feature budgets of $\pm2.0$---to maximise the number of samples classified as potable. 
\textbf{(2) Tree Planting~\cite{wood2023tree}.} Employed exclusively in Experiment~E5 for the subsolver comparison.  Four species placed on an \(n_{\mathrm{grid}}\times n_{\mathrm{grid}}\) plot (\(n_{\mathrm{grid}}=6\)).  Seven site features feed neural survival models. Dual decomposition (PGD and log-barrier) is compared to a Big-M baseline.

\begin{figure}[t]
  \centering
  \includegraphics[width=0.45\textwidth]{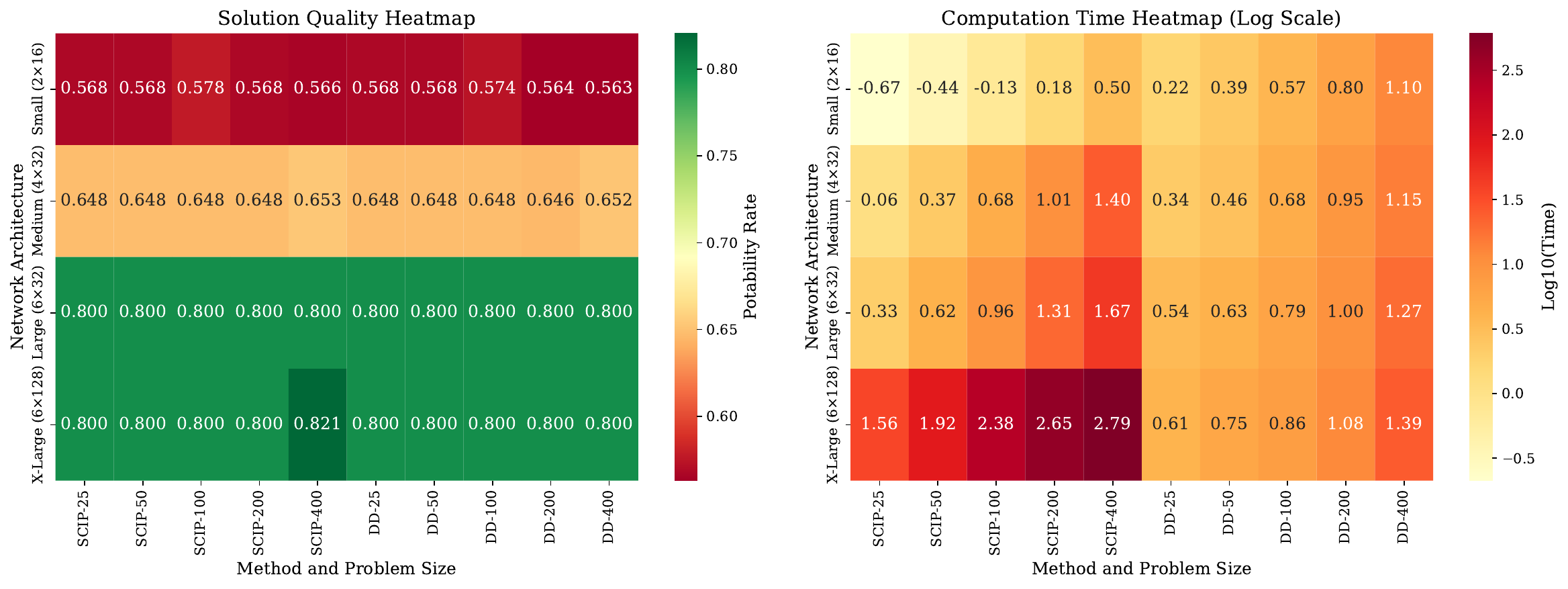}
  \caption{(E1) \textbf{Comparison with Linearisation-based Methods.} Solution quality (left) and computation time (right) across network architectures and problem sizes. Colours denote potability rate and $\log_{10}$ wall-clock time, respectively.}
  \label{fig:heatmap}
\end{figure}

\begin{figure}[t]
  \centering
  \includegraphics[width=0.45\textwidth]{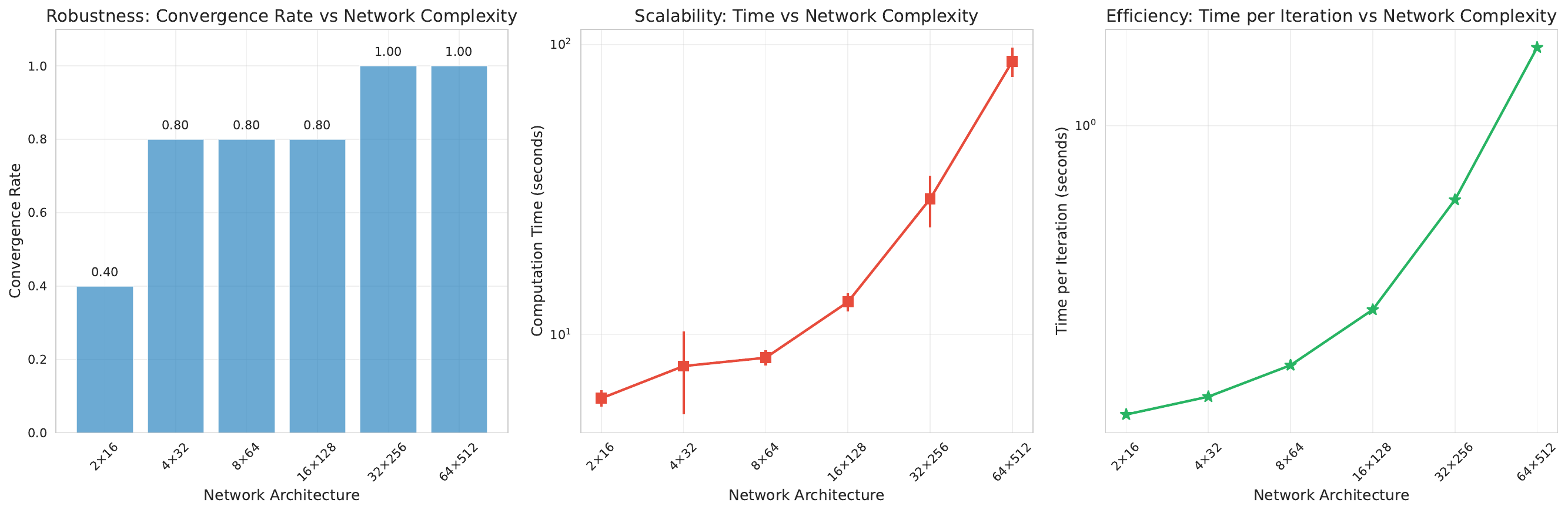}
  \caption{(E2) \textbf{Scalability stress test} with fixed $n=100$. \textbf{Left}: convergence rate; \textbf{centre}: total computation time (log scale); \textbf{right}: average time per iteration (log scale).}
  \label{fig:stress-test}
\end{figure}

\subsection{(E1) Comparison with Linearisation-based Methods}

We benchmark DD against the Big-M linearisation in \textsc{PySCIPOpt-ML} on four configurations—\emph{Small} (2×16), \emph{Medium} (4×32), \emph{Large} (6×64), and \emph{X-Large} (6×128)—and five instance sizes.

\noindent\textbf{Solution Quality.} DD aligns with the monolithic baseline across all configurations (Fig.~\ref{fig:heatmap}, left), evidencing virtually no performance loss from decomposition on this task.
\textbf{Computation Time.} Big-M runtimes exceed $10^{2}\,\mathrm{s}$ as depth or $n$ grows, whereas DD stays near \(10\,\mathrm{s}\) even on the hardest cases, delivering \(10\)–\(100\times\) speed-ups (Fig.~\ref{fig:heatmap}, right).

\subsection{(E2) Scalability Stress Test Across Network Depth}

We fix the sample size at \(n=100\) and sweep six increasingly deep MLPs
\footnote{\(2{\times}16,\;4{\times}32,\;8{\times}64,\;16{\times}128,\;32{\times}256,\;64{\times}512\) hidden units.}%
, repeating each configuration over five random seeds.  For every run we record
(i)~success rate within 50 DD iterations,
(ii)~total wall-clock time, and
(iii)~average time per iteration.

\textbf{Robustness.}  Only \(40\,\%\) of the \(2{\times}16\) runs converge within the 50-iteration cap, yet every network at or beyond \(16{\times}128\) succeeds in all trials (Fig.~\ref{fig:stress-test}, left).
\textbf{Scalability.} Total solve time grows with depth but remains under \(90\,\text{s}\) even for the \(64{\times}512\) model (Fig.~\ref{fig:stress-test}, centre).
\textbf{Efficiency.} Per-iteration cost increases sub-linearly with network size (Fig.~\ref{fig:stress-test}, right), implying that the extra overhead stems chiefly from neural evaluation, not the optimisation loop. For reference, \textsc{PySCIPOpt-ML} fails beyond the \(6{\times}128\) (X-Large) model: deeper nets either exceed memory limits or time out at \(300\,\text{s}\).  DD therefore avoids the prohibitive linearisation burden through its decomposition strategy.

\subsection{(E3) Ablation Study: Impact of Dual Coordination}\label{sec:abl_no_decomp}

To quantify the contribution of the dual–update mechanism we construct a degenerate variant—\emph{single-step gradient} (\textbf{SSG})—that performs one projected-gradient step on \(u\), heuristically rounds to an integer \(x\), and terminates.\footnote{Identical to Algorithm~1 but with the dual update removed and \(\lambda=\rho=0\).}  
The design is inspired by the direct gradient-descent solver of \textsc{DiffILO}~\cite{geng2025differentiable}, which tackles NN-embedded MIPs without dual variables, and the complete SSG pseudocode is given in Appendix. 
Powers-of-two instance sizes (\(n\in\{32,64,128\}\)) are chosen for clearer logarithmic plots.

\begin{figure}[t]
  \centering
  \includegraphics[width=0.45\textwidth]{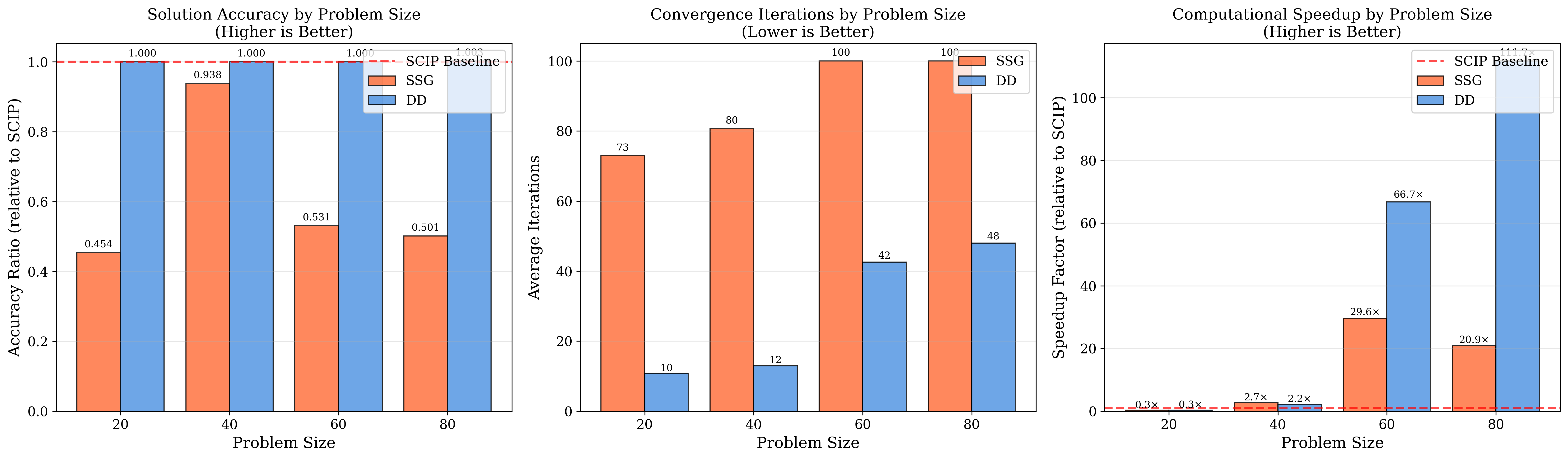}
  \caption{(E3) \textbf{Ablation results.}  SSG (no dual coordination) versus full DD.  Left: solution quality (normalised to SCIP).  Centre: \emph{log} average iterations to feasibility.  Right: speed-ups over SCIP (higher is better).}
  \label{fig:e3_ablation}
\end{figure}

\begin{itemize}
\item \textbf{Solution quality.}  
      SSG attains erratic accuracies—\(0.45\), \(0.94\), \(0.53\), and \(0.50\) of the SCIP benchmark as \(n\) grows—whereas DD stays at unity (and even \(0.998\) for \(n=80\)), confirming that the decomposition preserves optimality.

\item \textbf{Convergence.}  
      SSG requires \(73\!-\!100\) iterations (avg.\ \(\approx88\)), often stalling at the cap; DD converges in an average of \(10\), \(12\), \(42\), and \(48\) iterations for \(n=20,40,60,80\), respectively, underscoring the stabilising effect of the dual updates.

\item \textbf{Runtime.}  
      SSG is slower than SCIP for \(n=20\) (\(0.3\times\)), modestly faster for \(n=40\) (\(2.7\times\)), and reaches \(20\!-\!30\times\) speed-ups at larger scales.  
      DD starts on par with SSG at \(n=40\) (\(2.2\times\)), then widens the gap to \(66.7\times\) and \(111.7\times\) for \(n=60\) and \(80\), respectively.

\end{itemize}

Dual coordination is indispensable: removing it degrades numerical robustness, inflates iteration counts by two–three orders of magnitude, and erodes the runtime advantage as \(n\) increases.  Alternating dual updates are therefore critical for both rapid convergence and high-quality solutions.

\subsection{(E4) Adaptability Across Neural Architectures}

We run the \emph{unchanged} DD solver on four disparate predictors: a shallow fully connected network (\textbf{FC}), a lightweight \textbf{CONV1D}, a deep \textbf{RESNET}, and a sequential \textbf{LSTM}.  By contrast, linearisation frameworks such as \textsc{PySCIPOpt-ML} only support simple feed-forward layers and cannot encode CNN or LSTM backbones at all.

\begin{itemize}\itemsep2pt
  \item \textbf{Predictive fidelity.}  
        DD preserves the original accuracies: \(0.694\) (FC), \(0.743\) (CONV1D), \(0.791\) (LSTM), and \(0.696\) (RESNET), demonstrating zero distortion of the learned models.

  \item \textbf{Convergence across backbones.}  
        Within the 50-iteration limit DD satisfies \(\|u-x\|_\infty<10^{-4}\) in \(0.83\) of CONV1D runs, \(0.75\) of FC runs, \(0.83\) of RESNET runs, and \(0.50\) of LSTM runs—architectures that existing linearisation tool-chains cannot even represent.

  \item \textbf{Scalable runtime.}  
        At the largest problem size (\(n=200\)) average solve times are \(\approx6\,\text{s}\) (FC), \(\approx8\,\text{s}\) (RESNET), \(\approx15\,\text{s}\) (CONV1D), and \(\approx50\,\text{s}\) (LSTM), all well inside the \(300\,\text{s}\) wall-clock timeout where linearised formulations already fail.
\end{itemize}


\subsection{(E5) Modular Subsolver Test: PGD \textit{vs.} Log-Barrier}
\label{sec:subsolver_cmp}

We verify the plug-and-play nature of our framework by fixing all outer-loop parameters (\(\rho=10\), 50 iterations, stopping criterion) and only swapping the NN subsolver:
\textbf{DD-PGD} (25 Adam steps with projection) versus
\textbf{DD-Barrier} (damped Newton on a log-barrier with five CG steps, \(\mu\!\leftarrow\!0.1\mu\)).  
The monolithic linearisation baseline (\textsc{SCIP}) is included for reference.

Results (Fig.~\ref{fig:runtime_subsolver}) exhibit a three-tier structure: SCIP-ML escalates from 111 s to 860 s; DD-Barrier completes in 8.5s, 24.8s and 63.1s (13–16× speed-ups); DD-PGD runs in 0.4s, 1.4s and 6.5s (130–280× speed-ups over SCIP, 9–21× over Barrier). Preservation of identical classification accuracy confirms that these gaps stem purely from the continuous-solver choice. This isolated subsolver swap—without any change to the MIP block, dual updates or data handling—demonstrates true modularity: users can interchange NN optimisation routines to match constraint complexity (e.g.\ cheap projection vs.\ dense polyhedral limits) while retaining all theoretical and empirical benefits of the dual-decomposition framework.  

\subsection{Summary of Findings}
Across five studies we establish three main results.  \textbf{(i) Scalability}: DD maintains near-optimal quality while delivering \(10\!\times\!-\!280\!\times\) speed-ups over Big-M linearisation (E1, E2).   \textbf{(ii) Necessity of dual updates}: removing them degrades accuracy by up to \(55\%\) and inflates iteration counts by two orders of magnitude (E3).   \textbf{(iii) Modularity \& Adaptability}: DD functions as a plug-and-play layer—handling CNN, LSTM backbones and interchangeable PGD/Barrier sub-solvers without code changes—yet still outperforms linearisation baselines by at least an order of magnitude (E4, E5).   These empirical results corroborate the theoretical claims of Section~\ref{sec:introduction} and demonstrate that dual decomposition offers a practical path to embedding complex neural surrogates in large-scale mixed-integer optimisation.

\begin{figure}[t]
  \centering
  \includegraphics[width=0.45\textwidth]{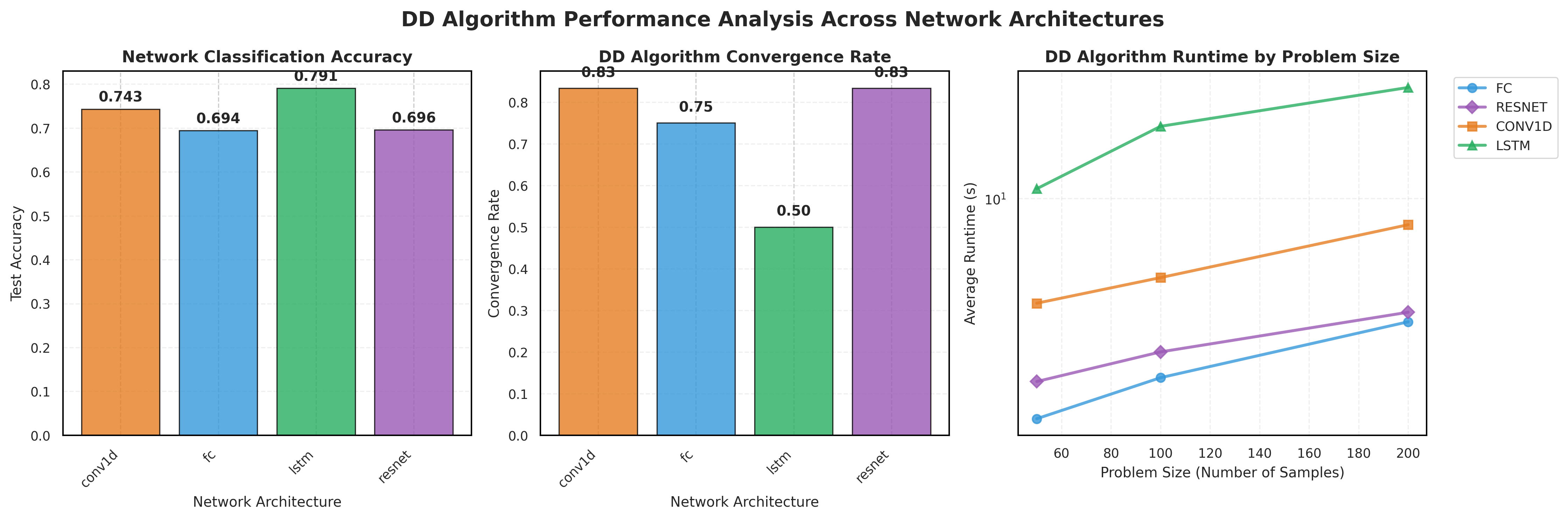}
  \caption{(E4) \textbf{Adaptability across architectures.}  
           Left: stand-alone test accuracy.  
           Centre: fraction of DD runs (five seeds) converging within the 50-iteration cap.  
           Right: geometric-mean runtime (log scale) for \(n\in\{60,100,200\}\).}
  \label{fig:modularity_arch}
\end{figure}

\begin{figure}[t]
  \centering
  \includegraphics[width=0.95\linewidth]{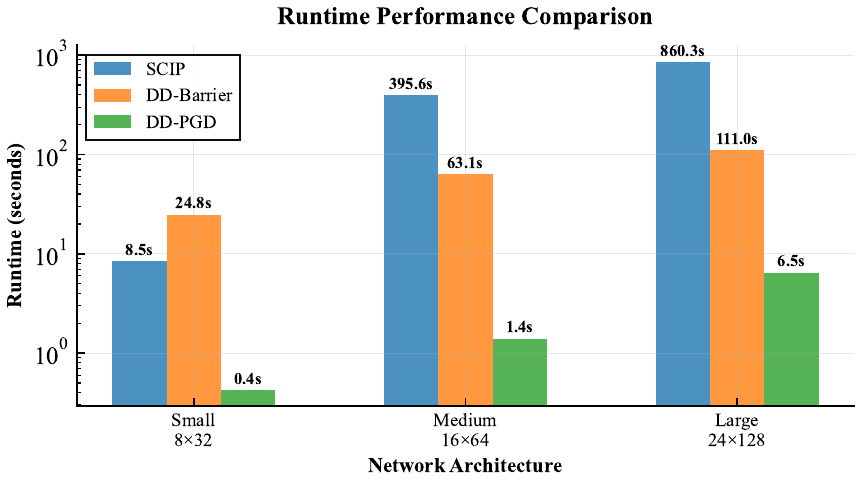}
  \vspace{-4pt}
  \caption{(E5) \textbf{Modular subsolver test.}  
           Wall–clock runtime (log scale) for DD-PGD, DD-Barrier, and the
           linearisation baseline on the \textit{Small} ($8{\times}32$),
           \textit{Medium} ($16{\times}64$), and \textit{Large}
           ($24{\times}128$) MLPs.}
  \label{fig:runtime_subsolver}
\end{figure}

%% file: sections/6conclusion.tex
\section{Conclusion}

We proposed a modular and scalable dual decomposition framework for integrating neural networks into mixed-integer optimization, enabling principled coordination between learning-based models and combinatorial solvers without requiring explicit constraint encodings. Our method outperforms monolithic baselines in both runtime and flexibility across a range of problem sizes and neural architectures. Despite its advantages, several limitations remain. First, the convergence rate can be sensitive to hyperparameters such as the penalty $\rho$ and step size, suggesting the need for adaptive or learned scheduling strategies. Second, our framework currently assumes access to gradient information from the neural network; extending to non-differentiable or black-box predictors remains an open challenge. Finally, while we focus on a single network, many real-world problems involve multiple interacting predictors or dynamic decision feedback, which could be explored via multi-block decomposition or hierarchical coordination.

%% file: sections/8appendix.tex
\subsection{Baseline Algorithms}\label{app:baselines}

\subsubsection{AutoLinearized: Exact Big--$M$ Formulation}

\paragraph{Rationale.}
AutoLinearized follows the \emph{exact embedding} paradigm: every nonlinear
activation is rewritten as a mixed-integer linear block, so the original
objective \emph{plus} the network becomes a single MILP.  A modern solver
can then certify global optimality.

\paragraph{Construction.}
For a ReLU network we introduce a continuous output variable
$z_{\ell}$ and a binary indicator $y_{\ell}$ for each hidden unit; the
ReLU $s=\max\{0,h\}$ is replaced by
\begin{align}
  &s\ge h,       & s &\le h-M_1(1-y),\\[-0.3em]
  &s\ge 0,       & s &\le M_2y, \tag{B.1}
\end{align}
where $M_1,M_2$ are interval bounds.  Constraints~(B.1) are added
layer-wise with \texttt{pyscipopt-ml.add\_predictor\_constr}.  We keep
SCIP~8.0.4 at default settings.

\paragraph{Pros \& Cons.}
The formulation is exact but may inject thousands of binaries, making
solve time explode for deep nets; nonetheless it provides a strong,
certifiable reference point.

\subsubsection{Single-Step Gradient (SSG) Heuristic}

\paragraph{Rationale.}
SSG drops integrality during optimisation: discrete variables are
relaxed to their continuous domain, smooth penalties enforce the
original constraints, and first-order updates run entirely in PyTorch.
A final rounding step returns a feasible integer solution.

\paragraph{Workflow.}
Algorithm~\ref{alg:ssg} summarises the procedure.
The relaxed vector $x$ is initialised inside its box bounds; in each
iteration we back-propagate through the network, take an Adam step on
the penalised objective, and project back to the box.  Whenever a fully
feasible iterate is encountered we record its objective value; the best
one is returned after $T$ iterations.  The method involves \emph{no}
external solver, so its cost is dominated by forward/backward passes.

\paragraph{Hyper-parameters.}
We fix $T=100$, learning-rate schedule
$\eta_t=10^{-2}\!\cdot\!(1+0.01t)^{-0.5}$, and penalty weights
$\mu_{\mathrm{lin}}=\mu_{\mathrm{nn}}=50$ in all experiments; a brief
grid search on a held-out validation set produced these values.

\begin{algorithm}[t]
\caption{Single-Step Gradient (SSG) Heuristic}\label{alg:ssg}
\begin{algorithmic}[1]
\Require initial $x^{0}$, bounds $\underline x,\overline x$,
         steps $\{\eta_t\}_{t=1}^{T}$,
         penalties $\mu_{\mathrm{lin}},\mu_{\mathrm{nn}}$
\Ensure best feasible pair $(\hat x,\hat O)$
\State $\hat O\leftarrow +\infty$, \quad $\hat x\leftarrow\varnothing$
\For{$t=1,\dots,T$}
    \State \textbf{Forward:} compute
      $p_{\text{lin}}\!=\!\max\{0,A_{\text{MIP}}x^{t-1}-b_{\text{MIP}}\}$,
      $p_{\text{nn}}\!=\!\max\{0,A_{\text{NN}}[x^{t-1};f_{\theta}(x^{t-1})]-b_{\text{NN}}\}$
    \State $\mathcal L\!\leftarrow\!
       c^{\top}x^{t-1}+d^{\top}f_{\theta}(x^{t-1})
       +\mu_{\text{lin}}\|p_{\text{lin}}\|_2^{2}
       +\mu_{\text{nn}}\|p_{\text{nn}}\|_2^{2}$
    \State \textbf{Backward:} $x^{t}\leftarrow
           \operatorname{Proj}_{[\underline x,\overline x]}
           \bigl(x^{t-1}-\eta_t\nabla_x\mathcal L\bigr)$
    \If{$p_{\text{lin}}=p_{\text{nn}}=0$ \textbf{and}
        $c^{\top}x^{t}<\hat O$}
        \State $\hat O\leftarrow c^{\top}x^{t}$,
               \quad $\hat x\leftarrow\mathrm{round}(x^{t})$
    \EndIf
\EndFor
\State \Return $(\hat x,\hat O)$
\end{algorithmic}
\end{algorithm}

\paragraph{Pros \& Cons.}
SSG is light-weight and highly parallelisable but provides no optimality
certificate and may return sub-optimal solutions if the penalty weights
are not well tuned.

\subsection{Datasets}
\subsubsection{Water Potability Benchmark}
\label{app:water-potability}
This benchmark corresponds to the \texttt{WP} problem type in \textsc{SurrogateLIB}\cite{turner2023pyscipopt}—a library of mixed-integer programs that embed machine-learning predictors as surrogate constraints.
This benchmark portrays the work of a public–health agency that must transform a collection of unsafe water samples into drinkable ones while operating under strict treatment budgets.  Each raw sample is described by a nine–dimensional feature vector
\(
w_i \in \mathbb{R}^9,\; i\in[n],
\)
where \([n]\) denotes the index set of the initially non-potable samples. 

\paragraph{Decision variables.}
For every sample \(i\) and attribute \(j\) we introduce
\begin{itemize}
    \item \(a_{i,j}\ge 0\): upward adjustment applied to attribute \(j\),
    \item \(b_{i,j}\ge 0\): downward adjustment applied to attribute \(j\),
    \item \(x_{i,j}\): value of attribute \(j\) after treatment,
    \item \(y_i\in\{0,1\}\): indicator that equals \(1\) if sample \(i\) is potable post-treatment.
\end{itemize}

\paragraph{Budgets.}
Each attribute \(j\in[9]\) has independent non-negative budgets
\(
\gamma_j^{+},\gamma_j^{-}
\)
limiting the total upward (\(a\)) and downward (\(b\)) changes across all samples.

\paragraph{Potability oracle.}
A pre-trained classifier \(f:\mathbb{R}^9\!\rightarrow\!\{0,1\}\) returns \(1\) when its input vector is judged drinkable.

\paragraph{Mixed-integer formulation.}
The optimisation problem is
\begin{equation}
\label{eq:mip-potability}
\begin{aligned}
\max_{a,b,x,y}\quad &\sum_{i=1}^{n} y_i \\[4pt]
\text{s.t.}\qquad
&x_{i,j} = w_{i,j} + a_{i,j} - b_{i,j} && \forall i\in[n],\, j\in[9], \\[2pt]
&\sum_{i=1}^{n} a_{i,j} \;\le\; \gamma_j^{+}      && \forall j\in[9], \\[2pt]
&\sum_{i=1}^{n} b_{i,j} \;\le\; \gamma_j^{-}      && \forall j\in[9], \\[2pt]
&y_i = f(x_i)                                      && \forall i\in[n], \\[2pt]
&a_{i,j},\,b_{i,j} \ge 0,\quad y_i\in\{0,1\}. &&
\end{aligned}
\end{equation}

The objective maximises the count of samples deemed potable while ensuring that the cumulative treatment applied to each attribute stays within the allotted budgets.  Data for this problem is publicly available\footnote{\url{https://github.com/MainakRepositor/Datasets/tree/master}}.

\subsubsection{Tree Planting Benchmark}
\label{app:tree-planting}

This benchmark (\texttt{TP} in \textsc{SurrogateLIB}~\cite{turner2023pyscipopt}) models the work of a forestry agency that must re-plant a cleared strip of land.  
The agency wishes to \emph{maximise the total expected number of trees that survive} while  
(i) respecting a limited budget for soil sterilisation and planting costs and  
(ii) guaranteeing a minimum expected survival target for every tree species considered.

\paragraph{Instance data.}
A grid of \(n\) candidate planting locations is generated; each location \(i\in[n]\) is described by a seven-dimensional feature vector
\(
x_i\in\mathbb{R}^7
\)
that includes soil, light and moisture indicators as well as a \emph{sterilised} flag (initially~\(0\)).  
Four tree species are available, indexed by \(k\in[4]\).  For every species \(k\) we are given
a minimum expected survival requirement \(\gamma_k\in\mathbb{R}_{\ge0}\),
a planting cost \(c_k\in\mathbb{R}_{\ge0}\),
and a trained surrogate predictor
\(
f_k : \mathbb{R}^7 \to [0,1]
\)
returning the probability that species \(k\) survives at a given location.

\vspace{4pt}
\paragraph{Decision variables.}  For each location \(i\) and species \(k\):
\begin{itemize}
    \item \(p_{i,k}\in\{0,1\}\): \(1\)~iff species \(k\) is planted at location \(i\);
    \item \(s_{i,k}\in[0,1]\): predicted survival probability of \(k\) at \(i\) (output of the surrogate);
    \item \(s'_{i,k}\in[0,1]\): survival contribution that is counted in the objective;
\end{itemize}
and for every location
\(
z_i\in\{0,1\}
\)
indicates whether its soil is sterilised.

\paragraph{Budgets and targets.}
\begin{itemize}
    \item Sterilisation budget \(\beta\in\mathbb{N}\) limits the number of sites that can be sterilised.
    \item Planting budget \(B\in\mathbb{R}_{\ge0}\) bounds the total cost of saplings.
    \item Species-wise targets \(\gamma_k\) ensure biodiversity.
\end{itemize}

\paragraph{Mixed-integer formulation.}
\begin{equation}
\label{eq:mip-tree}
\begin{aligned}
\max_{p,s,s',z}\quad
& \sum_{i=1}^{n}\sum_{k=1}^{4} s'_{i,k} \\[2pt]
\text{s.t.}\quad
& \sum_{i=1}^{n} s'_{i,k} \;\ge\; \gamma_k && \forall k\in[4] \\[2pt]
& p_{i,k}=0 \;\Longrightarrow\; s'_{i,k}=0 && \forall i\in[n],\,k\in[4] \\[2pt]
& p_{i,k}=1 \;\Longrightarrow\; s'_{i,k}\le s_{i,k} && \forall i\in[n],\,k\in[4] \\[2pt]
& \sum_{k=1}^{4} p_{i,k}=1                         && \forall i\in[n] \\[2pt]
& \sum_{i=1}^{n}\sum_{k=1}^{4} c_k p_{i,k}\;\le\; B \\[4pt]
& \sum_{i=1}^{n} z_i \;\le\; \beta \\[2pt]
& s_{i,k}=f_k\!\bigl(x_i,z_i\bigr)                && \forall i\in[n],\,k\in[4] \\[2pt]
& p_{i,k},z_i\in\{0,1\},\; s'_{i,k}\in[0,1]. &&
\end{aligned}
\end{equation}

\subsection{Convergence Proof for Theorem~\ref{thm:stationary}}
\label{app:conv}

We reproduce the augmented-Lagrangian merit
\[
  \mathcal{M}_{\rho}(x,u,\lambda)
  := g(u)+c^{\top}x
     +\frac{\rho}{2}\Bigl\|u-x+\frac{\lambda}{\rho}\Bigr\|_{2}^{2}
     -\frac{1}{2\rho}\|\lambda\|_{2}^{2},
\]
where $g(u)=d^{\top}f_{\theta}(u)$.
Under Assumption~\ref{ass:regularity} the following properties hold.

\begin{lemma}[Strong convexity of the $u$-block]
\label{lem:strong}
For $\rho>L$ the map
$u\!\mapsto\!g(u)+\tfrac{\rho}{2}\|u-x\|_2^{2}$
is $(\rho-L)$-strongly convex and $(\rho+L)$-smooth.
\end{lemma}
\begin{proof}
$g$ is $L$-smooth and convex; adding
$\tfrac{\rho}{2}\|u-x\|_2^{2}$ increases both curvature bounds by
$\rho$.  \qed
\end{proof}

\begin{lemma}[Per-iteration decrease]
\label{lem:descent}
Let $(x^{(k)},u^{(k)},\lambda^{(k)})$ be the $k$-th iterate.  
Then
\[
  \mathcal{M}_{\rho}^{(k+1)}
  \;\le\;
  \mathcal{M}_{\rho}^{(k)}
  -\frac{\rho-L}{2}\,\|u^{(k)}-x^{(k)}\|_{2}^{2}.
\]
\end{lemma}
\begin{proof}
\emph{(i) $u$-update.}  
By Lemma~\ref{lem:strong} and exact minimisation,
\(
  \mathcal{M}_{\rho}(x^{(k-1)},u^{(k)},\lambda^{(k-1)})
  \le
  \mathcal{M}_{\rho}(x^{(k-1)},u^{(k-1)},\lambda^{(k-1)})
  -\frac{\rho-L}{2}\|u^{(k)}-u^{(k-1)}\|_2^{2}.
\)

\noindent
\emph{(ii) $x$-update.}  
Since $\mathcal{X}$ is finite and the MIP subproblem is solved
optimally,
\(
  \mathcal{M}_{\rho}(x^{(k)},u^{(k)},\lambda^{(k-1)})
  \le
  \mathcal{M}_{\rho}(x^{(k-1)},u^{(k)},\lambda^{(k-1)}).
\)

\noindent
\emph{(iii) Dual update.}  
Setting
$\lambda^{(k)}=\lambda^{(k-1)}+\rho(u^{(k)}-x^{(k)})$
gives
\(
  \mathcal{M}_{\rho}(x^{(k)},u^{(k)},\lambda^{(k)})
  =
  \mathcal{M}_{\rho}(x^{(k)},u^{(k)},\lambda^{(k-1)})
  -\frac{\rho}{2}\|u^{(k)}-x^{(k)}\|_2^{2}.
\)
Combining (i)–(iii) and using
$\|u^{(k)}-u^{(k-1)}\|_2\le
 \|u^{(k)}-x^{(k)}\|_2+\|x^{(k)}-u^{(k-1)}\|_2$
yields the claimed inequality.  
\end{proof}

\paragraph{Proof of Theorem~\ref{thm:stationary}.}
Summing Lemma~\ref{lem:descent} over $k$ and noting that
$\mathcal{M}_{\rho}$ is bounded below
(by $-\|\lambda\|_2^{2}/(2\rho)$) proves
$\sum_{k}\|u^{(k)}-x^{(k)}\|_2^{2}<\infty$
and thus~\eqref{eq:lagrangian-rate}.
Boundedness follows from the coercivity of
$\mathcal{M}_{\rho}$ in $(x,u)$ and the finite cardinality of
$\mathcal{X}$.
A convergent subsequence exists by Bolzano–Weierstrass; continuity of
$\nabla_u \mathcal{L}_{\rho}$ and exact optimality of both subproblems
imply that every limit point satisfies the KKT conditions
\eqref{eq:kkt-conditions}. \qed

\subsection{Proof of Linear Scalability (Theorem~\ref{thm:scalability})}
\label{app:scale}

We detail the cost of one outer iteration.

\paragraph{MIP step.}
The discrete block~\eqref{eq:mip-update} contains
$p$ integer variables and $m_1$ linear constraints,
independent of the NN width or depth.  Let
$T_{\mathrm{MIP}}$ denote the time needed by the
branch-and-bound solver; this is treated as an oracle cost.

\paragraph{NN step.}
Evaluating $f_{\theta}(u)$ requires one forward pass
through the network; computing a vector–Jacobian product
for $g$ needs one reverse pass.  For dense or convolutional layers both
passes are $\Theta(P)$ flops \citep{goodfellow2016deep}.
The projection (PGD) or the Newton barrier uses
only the $m_{2}$ affine rows of $A_{\mathrm{NN}}$,
whose dimension is fixed by the application, hence $O(1)$ w.r.t.\ $P$.

\paragraph{Total.}
Adding the two independent costs yields
\(
  T_{\text{iter}}=\Theta\!\bigl(T_{\mathrm{MIP}}+P\bigr),
\)
proving Theorem~\ref{thm:scalability}. \qed